\documentclass{amsart}

\usepackage{amsmath,amssymb,enumerate}
\usepackage{mathtools}

\usepackage{epsfig}
\usepackage{color}

\newtheorem{theorem}{Theorem}[section]

\newtheorem{proposition}[theorem]{Proposition}
\newtheorem{lemma}[theorem]{Lemma}

\newtheorem{definition}[theorem]{Definition}
\newtheorem{remark}[theorem]{Remark}

\newcommand{\occult}[1]{}

\usepackage{soul}

\newcommand\eps{\epsilon}

\newcommand\NN{{\mathbb N}}

\newcommand\RR{{\mathbb R}}

\newcommand\ZZ{{\mathbb Z}}

\newcommand{\ph}{\varphi}

\newcommand{\GGG}{\mathcal{G}}

\newcommand{\DDD}{\mathcal{D}}

\newcommand{\SSS}{\mathcal{S}}
\newcommand{\PPP}{\mathcal{P}}

\newcommand{\Lspan}{\Lambda^\mathrm{span}}
\newcommand{\Lsep}{\Lambda^\mathrm{sep}}

\newcommand{\Pexp}{P_\mathrm{exp}^\perp}

\begin{document}

\title[Equilibrium states]{Equilibrium states for natural extensions of non-uniformly expanding local homeomorphisms}

\begin{abstract} 
 We examine uniqueness of equilibrium states for the natural extension of a topologically exact, non-uniformly expanding, local homeomorphism with a H\"older continuous potential function.  We do this by applying general techniques developed by Climenhaga and Thompson, and show there is a natural condition on decompositions that guarantees that a unique equilibrium state exists.  We then show how to apply these results to partially hyperbolic attractors.

\end{abstract}

\author{Todd Fisher and Krerley Oliveira}

\address{T.~Fisher, Department of Mathematics, Brigham Young University, Provo, UT 84602, USA, \emph{E-mail address:} \tt{tfisher@mathematics.byu.edu}}
\address{K.~Oliveira, Instituto de Matem\'atica, UFAL 57072-090 Macei\'o, AL, Brazil, \emph{E-mail address:} \tt{krerley@gmail.com}}

\thanks{T.F.\ is supported by Simons Foundation Grant \# 239708.  K.O.\ is partially supported by CNPq, CAPES, FAPEAL, INCTMAT and Foundation Louis D}


\maketitle

\section{Introduction}


For a dynamical system $f:X\to X$, where $f$ is continuous and $X$ is a compact metric space, and a continuous {\it potential function} $\ph:X\to \mathbb{R}$, the {\it topological pressure} is  $P(\ph; f)=\sup_{\mu}(h_\mu(f) + \int \ph d\mu)$ where the supremum is taken over all $f$-invariant Borel probability measures. An $f$-invariant Borel probability measure that achieves the supremum is an {\it equilibrium state}.  When $\ph= 0$ then the topological pressure is simply the entropy, and an equilibrium state is a measure of maximal entropy.  A core question in thermodynamic formalism is when there is a unique equilibrium state for $(X,f)$ and $\ph$.

Bowen
\cite{Bow75} proved that a homeomorphism of a compact metric space has a unique equilibrium state provided the system satisfies two conditions (expansivity and specification), and the potential function satisfies a dynamical bounded variation property, that is now referred to as the Bowen property.  
This result can be applied if $f:M\to M$ is a diffeomorphism, $\Lambda\subset M$ is a compact $f$-invariant hyperbolic basic set, and $\ph:M\to \mathbb{R}$ is H\"older continuous; so there is a unique equilibrium state for $f|_{\Lambda}$ and $\ph|_{\Lambda}$.

For diffeomorphisms there are a number of results outside of the uniformly hyperbolic setting.  Climenhaga and Thompson recently extended Bowen's techniques to a nonuniform version \cite{CT1, CT}, and these results have been applied for weak forms of hyperbolicity \cite{CFT_BV, CFT_Mane}. Using other techniques there are a number  of results for certain weak forms of hyperbolicity  \cite{BFT, CPZ, CrisostomoThazibi, RS}.

Sarig \cite{Sarig13} developed new symbolic tools for  $C^{1+\alpha}$ diffeomorphisms of surfaces with positive topological entropy using countable Markov chains. Buzzi, Crovisier, and Sarig \cite{BCS} then proved uniqueness of maximal entropy measures for transitive $C^\infty$ surface  diffeomorphisms with  positive topological entropy. Adapting these techniques Obata \cite{ObataStandard} recently proved for sufficiently large parameters the standard map has a unique measure of maximal entropy.

For continuous surjective maps of a compact metric space there are similarly a number of results on the existence and uniqueness of equilibrium states.  Walters \cite{Walters78} investigated properties of equilibrium states for uniformly expanding maps that are local homeomorphisms; his techniques use the transfer operator.

There are a number of results on existence and uniqueness of equilibrium states for non-uniformly expanding local homeomorphisms of a compact metric space; see for instance \cite{OV, RV, VV}.  Similar to the results in \cite{Walters78} these results use the transfer operator.  Ramos and Viana \cite{RV} show there is a condition, related to non-uniform hyperbolicity, such that if the potential is H\"older continuous and satisfies this condition, then there are finitely many ergodic equilibrium states for the potential.  Furthermore, they prove if the non-uniformly expanding map is topologically exact, then the equilibrium state is unique.

We examine the natural extension of a topologically exact local homeomorphism and obtain a condition that guarantees a unique equilibrium state.  We use the results of Climenhaga and Thompson \cite{CT} to establish the existence and uniqueness of equilibrium states for the natural extensions.  The main idea is that if there are `enough' points (in terms of pressure) with expansivity and specification, and the potential function has the Bowen property on these points, then there is a unique equilibrium state.

Our main result result (Theorem \ref{thm.unique})  shows that if there is a decomposition of orbit segments, as defined in Section \ref{s.background}, where the `bad' orbit segments for a parameter $\sigma$ (representing hyperbolicity by $\sigma$) have pressure less than the overall pressure of the system, then there is a unique equilibrium state.

This work is an extension of \cite{FO1}, where we examined unique equilibrium states for certain partially hyperbolic attractors that are topologically conjugate to natural extensions of a non-uniformly expanding local diffeomorphism.  Theorem \ref{t.generalAttractor} is an extension of Theorem \ref{thm.unique} for smooth systems, and the results in \cite{FO1} satisfy the hypotheses of \ref{t.generalAttractor} and give applications where this theorem can be applied. 

 \section{Background}\label{s.background}

In this section we review the results of Climenhaga and Thompson in \cite{CT} that we will use to establish the uniqueness of equilibrium states.

\subsection{Pressure}\label{ss.pressure}

Let $f\colon X\to X$ be a continuous map on a compact metric space.  We identify $X\times \mathbb{N}$ with the space of finite orbit segments by identifying $(x,n)$ with $(x,f(x),\dots,f^{n-1}(x))$.

Given a continuous potential function $\varphi\colon X\to \mathbb{R}$, write
$$
S_n\varphi(x) = S_n^f \ph(x) = \sum_{k=0}^{n-1} \varphi(f^k(x)).
$$
The \emph{$n^\mathrm{th}$ Bowen metric} associated to $f$ is defined by
\[
d_n(x,y) = \max \{ d(f^k(x),f^k(y)) \,:\, 0\leq k < n\}.
\]
Given $x\in X$, $\eps>0$, and $n\in \NN$, the \emph{Bowen ball of order $n$ with center $x$ and radius $\eps$} is
$
\Gamma^n_\eps(x) = \{y\in M \, :\,  d_n(x,y) < \eps\}.
$
A set $E\subset M$ is $(n,\eps)$-separated if $d_n(x,y) \geq \eps$ for all $x,y\in E$.

Given $\mathcal{D}\subset X\times \mathbb{N}$, we interpret $\DDD$ as a \emph{collection of orbit segments}. Write $\mathcal{D}_n = \{x\in X \, :\,  (x,n)\in \mathcal{D}\}$ for the set of initial points of orbits of length $n$ in $\mathcal{D}$.  Then we consider the partition sum
$$
\Lambda^{\mathrm{sep}}_n(\DDD,\ph,\epsilon; f) =\sup
\Big\{ \sum_{x\in E} e^{S_n\ph(x)} \, :\,  E\subset \mathcal{D}_n \text{ is $(n,\epsilon)$-separated} \Big\}.
$$
The \emph{pressure of $\ph$ on $\DDD$ at scale $\eps$} is 
$$
P(\mathcal{D},\ph,\epsilon; f) = \varlimsup_{n\to\infty} \frac 1n \log \Lambda^{\mathrm{sep}}_n(\mathcal{D},\ph,\epsilon),
$$
and the \emph{pressure of $\varphi$ on $\mathcal{D}$} is
$$
P(\mathcal{D},\ph; f) = \lim_{\epsilon\to 0}P(\mathcal{D},\ph,\epsilon).
$$

Given $Z \subset X$, let $P(Z, \varphi, \epsilon; f) := P(Z \times \NN, \varphi, \epsilon; f)$; observe that $P(Z, \varphi; f)$ denotes the usual upper capacity pressure \cite{Pesin}. We often write $P(\ph)$ in place of $P(X, \ph;f)$ for the pressure of the whole space. When $\ph=0$, our definition gives the \emph{entropy of $\mathcal{D}$}:
\begin{equation}\label{eqn:h}
\begin{aligned}
h(\mathcal{D}, \epsilon; f)= h(\mathcal{D}, \epsilon) &:= P(\mathcal{D}, 0, \epsilon) \mbox{ and } h(\mathcal{D})= \lim_{\epsilon\rightarrow 0} h(\mathcal{D}, \epsilon).
\end{aligned}
\end{equation}

Write $\mathcal{M}(f)$ for the set of $f$-invariant Borel probability measures and $\mathcal{M}_e(f)$ for the set of ergodic measures in $\mathcal{M}(f)$.
The variational principle for pressure \cite[Theorem 10.4.1]{VO} states that 
\[
P(\varphi)=\sup_{\mu\in \mathcal{M}(f)}\left\{ h_{\mu}(f) +\int \varphi \,d\mu\right\}
=\sup_{\mu\in \mathcal{M}_e(f)}\left\{ h_{\mu}(f) +\int \varphi \,d\mu\right\}.
\]

 \subsection{Obstructions to expansivity, specification, and regularity}

We recall definitions and results from \cite{CT}, which show that non-uniform versions of Bowen result in \cite{Bow75} suffice to prove uniqueness.

Given a map $f\colon X\to X$, the \emph{infinite Bowen ball around $x\in X$ of size $\eps>0$} is the set
\[
\Gamma^+_\eps(x) := \{y\in X \, :\,  d(f^k(x),f^k(y)) < \eps \text{ for all } n\geq 0 \}.
\]
If there exists $\eps>0$ for which $\Gamma^+_\eps(x)= \{x\}$ for all $x\in X$, we say $(X, f)$ is \emph{expanding}. 
For $f\colon X\to X$ a homeomorphism, the \emph{bi-infinite Bowen ball around $x\in X$ of size $\eps>0$} is the set
\[
\Gamma_\eps(x) := \{y\in X \, :\,  d(f^k(x),f^k(y)) < \eps \text{ for all } n\in \ZZ \}.
\]
If there exists $\eps>0$ for which $\Gamma_\eps(x)= \{x\}$ for all $x\in X$, we say $(X, f)$ is \emph{expansive}. 

\begin{definition} \label{almostexpansive}
For $f\colon X\rightarrow X$ a homeomorphism the set of non-expansive points at scale $\epsilon$ is  $\mathrm{NE}(\epsilon):=\{ x\in X \, :\,  \Gamma_\eps(x)\neq \{x\}\}$.  An $f$-invariant measure $\mu$ is  almost expansive at scale $\epsilon$ if $\mu(\mathrm{NE}(\epsilon))=0$.  Given a potential $\varphi$, the pressure of obstructions to expansivity at scale $\epsilon$ is
\begin{align*}
\Pexp(\varphi, \epsilon) &=\sup_{\mu\in \mathcal{M}_e(f)}\left\{ h_{\mu}(f) + \int \varphi\, d\mu\, :\, \mu(\mathrm{NE}(\epsilon))>0\right\} \\
&=\sup_{\mu\in \mathcal{M}_e(f)}\left\{ h_{\mu}(f) + \int \varphi\, d\mu\, :\, \mu(\mathrm{NE}(\epsilon))=1\right\}.
\end{align*}
This is monotonic in $\eps$, so we can define a scale-free quantity by
\[
\Pexp(\varphi) = \lim_{\epsilon \to 0} \Pexp(\varphi, \epsilon).
\]
\end{definition}

\begin{definition} 
A collection of orbit segments $\mathcal{G}\subset X\times \mathbb{N}$ has the (W)-\emph{specification at scale $\epsilon$} if there exists $\tau\in\mathbb{N}$ 
and $k_0$ 
such that for every $\{(x_j, n_j)\, :\, 1\leq j\leq k\}\subset \mathcal{G}$ with $n_j>k_0$, there is a sequence of `gluing times'  $\tau_1, \dots, \tau_{k-1}\in \mathbb{N}$ with $\tau_i\leq \tau$ for all $1\leq i\leq k-1$ and a point $x$ such that for $s_j=\sum_{i=1}^j n_j + \sum_{i=1}^{j-1} \tau_i$ and $s_0=\tau_0=0$ we have
$$
d_{n_j}(f^{s_{j-1} +\tau_{j-1}}(x), x_j)<\epsilon \textrm{ for every }1\leq j\leq k.
$$

\end{definition}

The above definition says that there is some point $x$ whose trajectory stays $\epsilon$ close to each of the $(x_i,n_i)$ in turn, taking a transition time of $\tau_i\leq \tau$ between each one.   We note that specification is related to, but distinct from, shadowing. 

\begin{definition} \label{Bowen}
Given $\mathcal{G}\subset X\times \mathbb{N}$, a potential $\varphi$  has the \emph{Bowen property on $\GGG$ at scale $\epsilon$} if
\[
V(\GGG,\ph,\epsilon) := \sup \{ |S_n\varphi (x) - S_n\varphi(y)| : (x,n) \in \GGG, y \in B_n(x, \epsilon) \} <\infty.
\]
We say $\varphi$ has the \emph{Bowen property on $\GGG$} if there exists $\epsilon>0$ so that $\varphi$ has the Bowen property on $\GGG$ at scale $\epsilon$.
\end{definition}

Note that if $\GGG$ has the Bowen property at scale $\eps$, then it has it for all smaller scales. 


\subsection{General results on uniqueness of equilibrium states}

Our main tool for existence and uniqueness of equilibrium states is \cite[Theorem 5.5]{CT}.  Before stating this result we need the next definition.

\begin{definition}
A \emph{decomposition} for $(X,f)$ consists of three collections $\mathcal{P}, \mathcal{G}, \mathcal{S}\subset X\times (\NN\cup\{0\})$ and three functions $p,g,s\colon X\times \mathbb{N}\to \NN\cup\{0\}$ such that for every $(x,n)\in X\times \NN$, the values $p=p(x,n)$, $g=g(x,n)$, and $s=s(x,n)$ satisfy $n = p+g+s$, and 
\begin{equation}\label{eqn:decomposition}
(x,p)\in \mathcal{P}, \quad (f^p(x), g)\in\mathcal{G}, \quad (f^{p+g}(x), s)\in \mathcal{S}.
\end{equation}
\end{definition}

Note that the symbol $(x,0)$ denotes the empty set, and the functions $p, g, s$ are permitted to take the value zero. 
 
\begin{theorem}[Theorem 5.5 of \cite{CT}]\label{t.generalM}
Let $X$ be a compact metric space and $f\colon X\to X$ a homeomorphism. 
Let $\ph \colon X\to\RR$ be a continuous potential function.
Suppose that $\Pexp(\ph) < P(\ph)$, and that $(X,f)$ admits a decomposition $(\PPP, \GGG, \SSS)$ with the following properties:
\begin{enumerate}
\item $\GGG$ has (W)-specification at any scale;
\item  $\ph$ has the Bowen property on $\GGG$;
\item $P(\PPP \cup \SSS,\ph) < P(\ph)$.
\end{enumerate}
Then there is a unique equilibrium state for $\ph$.
\end{theorem}

\section{Decompositions for nonuniformly expanding local homeomorphisms}\label{s.decompforg}
 
 We first describe the class of maps we want to investigate.  
 Let $(X, d_X)$ be a compact metric space and $g:X\to X$ a local homeomorphism.  Suppose that every inverse branch $g^{-1}$ is locally Lipschitz continuous; so there exists a bounded function $\sigma:X\to \mathbb{R}^+$ such that for each $x\in X$ there exists a neighborhood $U_x$ of $x$ such that $g_x:=g|_{U_x}:U_x \to g(U_x)$ is inveritble and for all $y, z\in g(U_x)$ we have
$$
d_X(g_x^{-1}(y), g_x^{-1}(z))\leq \sigma(x) d_X(y,z).
$$
We assume there exists some $\epsilon_0>0$ such that for all $x\in X$ we have $B_{\epsilon_0}(x)\subset U_x$ and $B_{\epsilon_0}(g (x))\subset g(U_x)$.
We also assume that $g$ is topologically exact, meaning for each open set $U\subset X$ there exists some $N$ such that $g^N(U)=X$.

 Although we will need to look at a decomposition on the natural extension we first describe a decomposition for the local homeomorphism $g$ and prove some properties for this decomposition and this map.  We cannot apply Theorem \ref{t.generalM} for the local homeomorphism since the map needs to be a homeomorphism to apply the theorem.  However, the estimates we obtain for the local homeomorphism will be used for the natural extension in the next section.
 

 Let $\sigma\in (0,1)$, $n\in\mathbb{N}$, and $j\in \{0,..., n-1\}$ be fixed.  Define
\begin{equation}\label{NUE}
\Sigma_\sigma^{j,n} =\{x\in X\,:\, \frac{1}{n-j} \sum\limits_{i=j}^{n-1} \log \sigma(g^ix) < \log\sigma\}.
\end{equation}

The next lemma now follows from the expansion estimates for points in $\Sigma_\sigma^{0,n}$.

\begin{lemma}\label{lem.expansion}
If there exists a sequence $n_k\to \infty$ such that $x\in \Sigma_\sigma^{0, n_k}$ for each $k$, then $\Gamma^+_\epsilon(x)=\{x\}$ for $\epsilon\leq \epsilon_0$.
\end{lemma}

\begin{proof}  Let $\epsilon\leq \epsilon_0$, $y\in \Gamma^+_\epsilon(x)$, and $n_k\in \mathbb{N}$ such that $x\in \Sigma_\sigma^{0, n_k}$.  This implies that $d(x,y)\leq \sigma^{n_k}\epsilon_0$.  Now as there is a sequence of $n_k\to \infty$ such that $x\in \Sigma_\sigma^{0, n_k}$ we see that $d(x,y)=0$ and $\Gamma^+_\epsilon(x)=\{x\}$.\
\end{proof}

We now define the decomposition we want.

$$
\begin{array}{llll}
\GGG_\sigma=\{ (x,n)\in X\times \mathbb{N}\, :\, x\in \Sigma_\sigma^{j,n} \, \forall \, 0\leq j\leq n-1\},\\
\SSS_\sigma=\{ (x,n)\in X\times \mathbb{N}\, :\, x\notin \Sigma_\sigma^{0,n}\}.
\end{array}$$

The collection $\GGG_\sigma$ is chosen so that there is uniform contraction by $\sigma$ along the inverse branch from $g^nx$ to $x$.  
The definition of  the collection of  orbits $\GGG_\sigma$ is inspired by the analogous notion of hyperbolic times, as in \cite{Alves}. 

\begin{remark}\label{r.concatenation}
A simple argument shows the following concatenation property of $\GGG_\sigma$: if $(x,n)$ and $(f^n(x),m)$ are in $\GGG_\sigma$, then $(x,n+m)\in \GGG_\sigma$.  
\end{remark}

We now let $(x,n)$ be an orbit segment and define $s$ to be the smallest integer where $0\leq s\leq n$ such that $(g^s x, n-s)\in \SSS_\sigma$.  Then one can easily check that $(x, n-s)\in \GGG_\sigma$ and this defines a decomposition on the orbit segments (where $\PPP_\sigma$ and $p(x,n)$ are both trivial). 

\subsection{Specification}\label{ss.spec}
 
 We now show that $\GGG_\sigma$ has specification for $g$.  The statement and proof below are modifications of those given by Proposition 3.2 in \cite{FO1}.

\begin{proposition}\label{specforg} Given $\epsilon\leq \epsilon_0$  there exist $\tau=\tau(\epsilon)$ such that if $\{(x_j,n_j)\}_{j=0}^m \subset \GGG_\sigma$, then there exists $\tau_i \leq \tau$ for $i=1,\dots,l-1$ and  some $z \in N$ such that

$$
d(g^{m}(x_j),g^{m+r_{j-1}}(z)) \leq \epsilon,
$$ for $0\leq m \leq n_j+k-1$, where $r_{0}=0$ and $r_j = \sum_{i=1}^{j} (n_i+\tau_i)$ for each $j \geq 1$.
\end{proposition}

\begin{proof} 

We first notice that if $\epsilon\leq \epsilon_0$, then there exists some $\tau=\tau(\epsilon)\in\mathbb{N}$ such that for all $y$ we have $g^j(B_\epsilon(y))=X$ for some $j\leq \tau$.  This follows from compactness of $X$ and the fact $g$ is topologically exact.

Now for $\{(x_j, n_j)\}_{j=0}^m\subset \GGG_\sigma$ we know there exists a set of points $X_{m-1}\subset B_\epsilon(g^{n_{m-1}}(x_{m-1}))$ and $\tau_m\leq \tau$ such that $g^{\tau_m}(X_{m-1})$ is the image of $B_\epsilon(g^{n_m}(x_m))$ by the inverse branch of $g^{-n_m}_{g^{n_m}(x_m)}$.

Similarly, there is a nonempty set $X_{m-2}\subset B_\epsilon(g^{n_{m-2}}(x_{m-2}))$ and $\tau_{m-1}\leq \tau$ such that $g^{\tau_{m-1}}(X_{m-2})$ is the image of $X_{m-1}$ by the inverse branch $g^{-n_{m-1}}_{g^{n_{m-1}}(x_{m-1})}$.

Continuing inductively we see that $X_0$ is nonempty and pick $z\subset X_0$.  By the uniform contraction along the inverse branches in $\GGG_\sigma$ we see that $z$ satisfies the requirements and $\GGG_\sigma$ has specification at scale $\epsilon$ for $g$.

\end{proof}

%
%
%
%

\subsection{Pressure of obstructions to expansivity}\label{ss.obstructions}

From Lemma \ref{lem.expansion} we know that a point $x\in X$ satisfies $\Gamma^+_\epsilon(x)\neq \{x\}$  for $\epsilon\leq \epsilon_0$
if there exists some $K(x)\in \mathbb{N}$ such that $\frac{1}{n}\sum_{i=0}^{n-1}\log \sigma(g^i x)\geq \log \sigma,$ for all  $n\geq K(x)$.

Let 
$$
A = \{ x\in X\, :\, \exists K(x), \frac{1}{n}\sum_{i=0}^{n-1}\log \sigma(g^i x)\geq \log \sigma, \forall n\geq K(x)\}.
$$
So $A$ contains the nonexpansive points of scale $\epsilon$, but may also contain some expansive points.  The next result follows from a modification of the proof of Lemma 3.5 in \cite{CFT_BV}, and will be used to show an upper bound for the pressure of obstructions to expansivity for the natural extension.

\begin{lemma}\label{lem.expansiveobstruction}
If $\varphi:X\to \mathbb{R}$ is continuous and $\mu\in \mathcal{M}_e(g)$ with $\mu(A)>0$, then $h_\mu(g) + \int \varphi d\mu\leq P(\SSS_\sigma, \varphi)$.
\end{lemma}

\begin{proof}  Fix $k\in\mathbb{N}$ and $A_k=\{ x\in A\, :\, K(x)\leq k\}$.  Since $\mu(\sum_k A_k)>0$ we know there exists some $k$ such that $\mu(A_k)>0$.  For $n>k$ and $x\in A_k$ we see that $(x,n)\in \SSS_\sigma$.  Then for each $\delta>0$ we have
$$
\Lsep_n(A_k, \ph,\delta; g) \leq \Lsep_n(\SSS_\sigma,\ph,\delta;g).
$$
Let $\eta\in (0, \mu(A_k))$ and let
\[
s_n(\ph, \delta, \mu,\eta ; g)=\mathrm{inf}\left \{ \sum_{x\in E} \exp\{S^g_n\ph(x)\} :  \mu \left(\bigcup_{x\in E} \overline B_n(x, \delta) \right)\geq \eta \right \},
\]
where the infimum is over all finite subsets in $X$.  The version of Katok's entropy formula for pressure \cite{M88} gives
\[
h_{\mu}(g) + \int \ph\, d \mu=\lim_{\delta\rightarrow 0}\limsup_{n\rightarrow \infty}\frac{1}{n}\log s_n(\ph, \delta, \mu, \eta; g).
\]
Finally,
$$s_n(\ph, \delta, \mu, \eta; g) \leq \Lspan_n(A_k, \ph, \delta;g) \leq \Lsep_n(A_k, \ph, \delta;g) \leq \Lsep_n(\SSS_\sigma,\ph,\delta;g)$$
and this implies that 
\[
h_{\mu}(g)+\int \ph\, d \mu \leq P(\SSS_\sigma, \ph) = \lim_{\delta \to 0} P(\SSS_\sigma, \ph,  \delta).
\]

\end{proof}

%

\section{Equilibrium states for the natural extension}

Define the space 
$$
\hat{X}= \{ \hat x=(x_0, x_1, x_2,....)\in X^\mathbb{N}\, :\, g(x_{i+1})=x_i \forall i\geq 0\}.
$$
One can define a metric on the space $\hat X$ in a number of ways.  For instance, we can define a metric on $\hat{X}$ by 
\begin{equation}\label{eq.extensao}
\hat{d}\big(\hat x, \hat y)=\sum_{n=0}^\infty a^{-n} d_X(x_n,y_n)
\end{equation}
where $\hat x = (x_n)$, $\hat y = (y_n)$, and $a>1$.  Given the metric on $\hat X$ one can then define functions that are H\"older continuous with respect to the metric, but note that this class of functions depends on the metric.  We will use the above class of metrics in our arguments.

The {\it natural extension} of $g$ is the homeomorphism $\hat g: \hat X\to \hat X$ defined by 
$$
\hat g(\hat x)= \hat g(x_0, x_1, ...)=(g(x_0), x_0, x_1,...).
$$
There is a projection map $\hat \pi:\hat X \to X$ defined by 
$$
\hat \pi(\hat x)=\hat \pi(x_0, x_1, x_2,...)=x_0
$$
that is a continuous surjective map such that $\hat \pi \circ \hat g = g \circ \hat \pi$.

For the natural extension we see that 
$$
d_X(\hat \pi(\hat x), \hat \pi(\hat y))=d_X(x_0, y_0)\leq 
\hat d(\hat x, \hat y).$$
So $\hat\pi$ is Lipschitz with constant 1.

We now define a decomposition for orbits in $\hat X$.  We say $(\hat x, n)\in \hat \GGG_\sigma$ if $(x_0,n)\in \GGG_\sigma$ and $(\hat x, n)\in \hat \SSS_\sigma$ if $(x_0, n)\in \SSS_\sigma$.  Similarly, we define $s\in \{0,..., n\}$ such that $(\hat x, s)\in \hat \GGG_\sigma$ and $(\hat g^s(\hat x), n-s)\in \hat \SSS_\sigma$ if the same is true for the decomposition of $(x_0, n)$.

We are now able to state our main theorem.

\begin{theorem}\label{thm.unique}
Let $\hat \varphi:\hat X \to \mathbb{R}$ be H\"older continuous and $\sigma\in (0,1)$. If $P(\hat \SSS_\sigma, \hat \varphi)<P(\hat \varphi; \hat g)$, then $(\hat X, \hat g)$ has a unique equilibrium state for $\hat \varphi$.
\end{theorem}

We note that in Theorem 1 and 2 of \cite{RV} they obtain similar results for the non-uniformly expanding map $g$.  Here they define the set of non-uniformly hyperbolic points with hyperbolic constant bounded by $\sigma$.  The result is that if the pressure on the complement of these points is less than the pressure of the entire system than there are finitely many ergodic equilibrium states for any H\"older continuous potential, and if there is a point with a dense pre-orbit, then the equilibrium state is unique.

Some differences between the condition in Theorem \ref{thm.unique} and those in Theorem 1 and 2 of \cite{RV} is that we look at the pressure bounded by the orbit segments in $\hat \SSS_\sigma$ for the decomposition, and not the pressure on an invariant set as in \cite{RV}.  Additionally, our results are for the natural extension and not for the system $(X,g)$.

To prove the above result we show that
\begin{itemize}
\item $\hat \GGG_\sigma$ has specification for sufficiently small scales,
\item $\hat \GGG_\sigma$ has the Bowen property, and
\item the pressure of obstructions to expansivity are bounded from above by the pressure on $\hat \SSS_\sigma$.
\end{itemize}

\subsection{Specification}
We now show that $\hat \GGG_\sigma$ has specification for sufficiently small $\epsilon>0$.

\begin{proposition}\label{specforghat} For $\epsilon\leq \epsilon_0$ the set $\hat \GGG_\sigma$ has the specification property at scale $\epsilon$ for $\hat g$. 
\end{proposition}

\begin{proof}

Given $\epsilon>0$, take $\tau=\tau(\epsilon/2)$ as in Proposition~\ref{specforg}.   By the uniform contraction on the fiber and $\mathrm{diam}(X)<\infty$ we know there exists a $\tau_s=\tau_s(\epsilon/2)$ such that for all $\hat x\in \hat X$ with $\hat \pi(\hat x)=x_0$ there exists some $j\leq \tau_s$ such that
$$
\hat g^{-k}(\hat \pi^{-1}(x_0))\subset B_{\epsilon/2}(\hat g^{-k}(\hat x))\cap \hat \pi^{-1}(\hat \pi(\hat g^{-k}\hat x))$$
for all $k\geq j$.

Let $\hat \tau=\max\{ \tau(\epsilon), \tau_s(\epsilon)\}$.  Now arguing as in Proposition \ref{specforg} we can use $\hat \tau$ and show that $\hat \GGG_\sigma$ has specification at scale $\epsilon$ for $\hat g$.

%
%
%
%
%
%
%
%
%
%
%

\end{proof}

\subsection{Bowen property}\label{Bowen} Now, we check that $\hat{\GGG}_\sigma$ has the Bowen property at any sufficiently small scale. Observe that given any $\hat x, \hat y \in \hat X$, by the H\"older continuity of $\hat \varphi$ we know there exist constants $C>0$ and $\alpha\in (0,1)$ such that 
\begin{equation}\label{eq.Bowen}
 |S_n \hat\varphi (\hat x) - S_n \hat \varphi(\hat y)|\leq \sum\limits_{i=0}^{n-1} C\hat d(\hat{g}^i(\hat x),\hat{g}^i(\hat y))^\alpha.
\end{equation}

If $(\hat x, n) \in \hat \GGG_\sigma$, writing $\hat x = (x_0,x_1,\dots)$ and $\hat y = (y_0,y_1,\dots) \in \Gamma^n_\eps(\hat x) $ we know  by the contraction properties of the inverse branches of $g^k$ at orbits in $\GGG_\sigma$ that 
$$
d(g^k(x_0),g^k(y_0))\leq \sigma^{n-k} d(g^n(x_0),g^n(y_0))\leq \epsilon \sigma^{n-k}.
$$
Thus, we have that for every $0 \leq i< n$:

$$
\begin{array}{llll}
d(\hat{g}^i(\hat x),\hat{g}^i(\hat y)) &= \sum\limits_{k=0}^{i} \frac{d(g^k(x_0),g^k(y_0))}{a^{i-k}} + \frac{d(\hat x,\hat y)}{a^i}\\
& \leq 
\sum\limits_{k=0}^{i} \frac{\sigma^{n-(i-k)} \epsilon}{a^{k}} + \frac{\epsilon}{a^i}\\
& = \epsilon \left(\sigma^{n-i} \left( \sum_{k=0}^i (\frac \sigma a)^k\right) + \frac{1}{a^i}\right)\\
&\leq C (\sigma^{n-i} + \frac{1}{a^i}).
\end{array}$$ 
for some constant $C$ depending on $\epsilon$. Using \eqref{eq.Bowen} we have
$$
\begin{array}{llll}
|S_n \hat\varphi (\hat x) - S_n \hat \varphi(\hat y)|&\leq \sum\limits_{i=0}^{n-1} C\hat d(\hat{g}^i(\hat x),\hat{g}^i(\hat y))^\alpha\\
& \leq \sum\limits_{i=0}^{n-1} C (\sigma^{n-i} + \frac{1}{a^i})^\alpha < K,
\end{array}
$$ for some constant $K$ big enough.

\subsection{Pressure of obstructions to expansivity}

As in Section \ref{s.decompforg} we know that a point $\hat x\in \hat X$ satisfies $\Gamma_\epsilon(\hat x)\neq \{\hat x\}$  for $\epsilon\leq \epsilon_0$ implies there exists some $K(x_0)\in \mathbb{N}$ such that $\frac{1}{n}\sum_{i=0}^{n-1}\log \sigma(g^i x_0)\geq \log \sigma, \forall n\geq K(x_0)$.

Let 
$$
\hat A = \{ \hat x\in \hat X\, :\, \exists K(x_0), \frac{1}{n}\sum_{i=0}^{n-1}\log \sigma(g^i x_0)\geq \log \sigma, \forall n\geq K(x_0)\}.
$$
So $\hat A$ contains the nonexpansive points of scale $\epsilon$, but may be larger.  Then as in Section \ref{s.decompforg} we can modify the proof   of Lemma \ref{lem.expansiveobstruction} to obtain the next result.

\begin{lemma}\label{lem.expansiveobstructionforextension}
If $\hat \varphi:\hat X\to \mathbb{R}$ is continuous and $\mu\in \mathcal{M}_e(\hat g)$ with $\mu(\hat A)>0$, then $h_\mu(\hat g) + \int \hat\varphi d\mu\leq P(\hat \SSS_\sigma, \hat \varphi)$.
\end{lemma}

If $\mu(\mathrm{NE}(\epsilon))>0$ for some $\mu\in \mathcal{M}_e(\hat g)$, then $\mu(\hat A)>0$.  So we have the next result.

\begin{theorem}\label{t.pressureofexpansiveobstructionsforlieft}
For $\hat \varphi:\hat X\to \mathbb{R}$ continuous, $\Pexp(\hat \ph,\eps) \leq P(\hat \SSS_\sigma, \hat \ph)$.
\end{theorem}

This completes the proof of Theorem \ref{thm.unique} since the decomposition satisfies the necessary conditions from Theorem \ref{t.generalM}.

\section{Partially hyperbolic attractors}\label{s.partiallyhyperbolic}

In this section we investigate the situation where $(\hat X, \hat g)$ is topologically conjugate to an attractor for a diffeomorphism of a manifold.  We will assume that $(X, g)$ is conjugate to a smooth local diffeomorphism where $X$ is a compact manifold.  (This can be relaxed to the idea of branched manifold $X$ as defined by Williams \cite{Williams74}.) Then $\sigma(x)$ is related to $\|Dg^{-1}(x)\|$.  

We assume that attractor is partially hyperbolic where the stable direction is uniformly contracting, and the attractor is foliated by center-unstable manifolds.  This was the situation investigated in \cite{FO1}.

Let $M$ be a compact manifold and $f:M\to M$ a diffeomorphism onto its image such that there is a continuous surjection $\pi:M\to X$ where
$$
\pi\circ f=g\circ \pi.
$$
Given $y\in X$ we set $M_y=\pi^{-1}(y)$.  Therefore, $M=\bigcup_{y\in X}M_y$.  Note that $f(M_y)\subset M_{g(y)}$, each $M_y$ is compact, and there is a maximum diameter for the sets $M_y$.  
Assume there exists some $\lambda_s\in (0,1)$ 
such that
$$
d(f(x), f(y))\leq \lambda_s d(x,y)
$$
for all $x,y\in M_z$ and all $z\in X$, and such that $\lambda_s$ is less than any contraction for points in $g$.  So that the $M_y$ are local stable manifolds.

The set $\Lambda=\bigcap_{n=0}^\infty f^n(M)$ is an attractor and $\Lambda$ is compact and $f$-invariant.  This will be a partially hyperbolic attractor.  An invariant set  $\Lambda$ for a diffeomorphism $f\colon M\to M$ is \emph{(weakly) partially hyperbolic} if there is a $Df$-invariant splitting $TM=E^s\oplus E^c\oplus E^u$, where at least one of $E^s$ or $E^u$ is nontrivial,  and constants $N\in \NN$, $\lambda>1$ such that for every $x\in \Lambda$ and every unit vector $v^\sigma\in E^\sigma$ for $\sigma\in \{s, c, u\}$, we have
\begin{enumerate}
\item[(i)] $\lambda \|Df^N_x v^s\|<\|Df^N_x v^c\|<\lambda^{-1}\|Df^N_x v^u\|$, and
\item[(ii)] $\|Df^N_x v^s\|<\lambda^{-1}<\lambda<\|Df^N_x v^u\|$.
\end{enumerate}

A partially hyperbolic set for  $f$ admits \emph{stable and unstable foliations} $W^s$ and $W^u$, which are $f$-invariant and tangent to $E^s$ and $E^u$, respectively \cite[Theorem 4.8]{yP04}.  In our situation there exists a stable foliation and the local stable leaves are given by the $M_y$.  

We also need the following fact so that the metric on $M$ is related to the metric on $X$.  Given $x,y\in M$  there exist 
$f$-invariant 
holonomies $h_{\pi (x), \pi (y)}:M_{\pi (x)}\cap \Lambda\to M_{\pi (y)}\cap \Lambda$  and a constant $C\geq 1$ such that 
\begin{equation}\label{eqn.equivalence}
\begin{array}{cccc}\frac{1}{C}[d_X(\pi(x), \pi(y)) + d_M(h_{\pi(x), \pi(y)}(x), y)]\\
 \leq d_M(x,y)\\
\leq
C[d_X(\pi (x), \pi (y)) + d_M(h_{\pi (x), \pi (y)}(x), y)]
\end{array}
\end{equation}
where $d_M$, and $d_X$ are the metrics on $M$ and $X$ respectively.  Furthermore, we assume that the holonomies are invariant for $f$ so that 
$$f(h_{\pi (x), \pi (y)}(z))=h_{g(\pi (x)), g(\pi (y))}f(z).$$

The attractor $\Lambda$ can be described as ``solenoid-like" and is topologically conjugate to the natural extension of the system $(X, g)$.

Let $g:X\to X$ be a local homeomorphism satisfying the conditions described in Section \ref{s.decompforg} and $\Lambda$ a partially hyperbolic attractor as defined above.  
%
%
%
%
%
Define the map $h:\Lambda \rightarrow \hat X$ by $h(p)=\hat x$ where $f^{-j}(p)\in M_{x_j}$ with $x_j\in X$ for all $j\geq 0$.  Then $h$ is a topological conjugacy from $(\Lambda, f)$ to $(\hat X, \hat g)$.

%
%
%
%
%

So locally $\Lambda$ is a product of a neighborhood in $X$ by a Cantor set.  
This is similar to the standard construction of the solenoid, sometimes also called the Smale-Williams attractor \cite{Williams74}, as the inverse limit of the doubling map on the circle.

If $h^{-1}$ is H\"older continuous in the above, then 
given a function $\varphi:\Lambda\to \mathbb{R}$ and $h:\Lambda\to \hat X$ we could define $\hat \varphi (\hat x)=\varphi (h^{-1}(\hat x))$.  Then Theorem \ref{thm.unique} gives conditions to ensure a unique equilibrium state for $\hat \varphi$ and we can show this gives a unique equilibrium state for $\varphi$.  However, in general it is not clear that $h^{-1}$ is H\"older, and so we will modify the arguments from the previous sections to  partially hyperbolic attractors.

We assume the same properties for $g:X\to X$ as in Section \ref{s.decompforg} and define the same decomposition for $g:X\to X$.  As in Section \ref{s.decompforg} we use the decomposition for $g:X\to X$ to give us a decomposition for orbit segments in $\Lambda$.  For the decomposition of $\Lambda$ we let an orbit segment $(y,n)$ for $(\Lambda, f)$ be in $\bar \GGG_\sigma$ if and only if $(\pi (y), n)\in \GGG_\sigma$, and similarly we let $(y, n)\in \bar \SSS_\sigma$ for $(\Lambda, f)$ if and only if $(\pi (y), n)\in \SSS_\sigma$.  We also define the decomposition of an orbit segment $(x,n)$ by the decomposition of the image of the orbit segment for $(\pi(x), n)$.

We note that for a partially hyperbolic attractor the stable manifolds vary H\"older continuously so the map $\pi$ is H\"older continuous.
The next result is similar to Theorem \ref{thm.unique}.

\begin{theorem}\label{t.generalAttractor}
Let $\varphi:\Lambda\to \mathbb{R}$ be a H\"older continuous potential function. 
If there exists some $\sigma\in (0,1)$ such that $P(\bar{\SSS}_\sigma, \varphi)< P(\varphi;  f|_\Lambda)$, then there exists a unique equilibrium state for $(\Lambda, f)$ associated with $\ph$.  
\end{theorem}

\begin{proof}
We first prove the Bowen property for the orbit segments in $\bar\GGG_\sigma;$  to do this we modify the arguments from Section \ref{Bowen}.
Given any $\bar x, \bar y \in \Lambda$, by the H\"older continuity of $\varphi$ we know there exist constants $C_0>0$ and $\alpha\in (0,1)$ such that 
\begin{equation}\label{eq.BowenAtt}
 |S_n \ph (\bar x) - S_n \ph(\bar y)|\leq \sum\limits_{i=0}^{n-1} C_0 d_M(f^i(\bar x),f^i(\bar y))^\alpha.
\end{equation}

If $(\bar x, n) \in \bar \GGG_\sigma$ and $\bar y  \in \Gamma^n_\eps(\bar x) $ we know  by the contraction properties of the inverse branches of $g$ at orbits in $\GGG_\sigma$ that 
$$
d_X(g^k(\pi(\bar x)),g^k(\pi(\bar y)))\leq C\epsilon \sigma^{n-k}
$$
where $C$ is from \eqref{eqn.equivalence}.
Thus, from the above and \eqref{eqn.equivalence} we have that for every $0 \leq i< n$:
$$
d_M(f^i(\bar x), f^i(\bar y))\leq \epsilon \sigma^{n-i} + \lambda^i\epsilon.
$$
	
Using \eqref{eq.BowenAtt} we have
$$
\begin{array}{llll}
|S_n \ph (\bar x) - S_n \ph(\bar y)|&\leq \sum\limits_{i=0}^{n-1} C_0 d_M(f^i(\bar x),f^i(\bar y))^\alpha\\
& \leq \sum\limits_{i=0}^{n-1} C_0 (\epsilon \sigma^{n-i} + \lambda^i\epsilon)^\alpha \\
& \leq \sum_{i=0}^{n-1} C_0 \epsilon^\alpha(\sigma^{n-i} + \lambda^i)< K,
\end{array}
$$ for some constant $K$ big enough.

We can modify the proof of Lemma \ref{lem.expansiveobstruction} to show that $\Pexp(\ph;  f|_\Lambda)\leq P(\bar{\SSS}_\sigma, \ph) < P(\ph;  f|_\Lambda)$.  This follows since the partially hyperbolic attractor is uniformly contracting in the stable direction, so the non-expansive points are distinguished by the the projection.

Lastly, we see that we can modify Proposition \ref{specforghat} to show that $\bar{\GGG}_\sigma$ has specification for sufficiently small scales.  
Given $\epsilon>0$, take $\tau$ as in Proposition~\ref{specforg}.   By the uniform contraction on the fiber and $\max_{y\in\Lambda}\mathrm{diam}(M_y)<\infty$ we know there exists a $\tau_s$ such that for all $x\in \Lambda$ there exists some $j\leq \tau_s$ such that
$$
f^{-k} (M_x)\subset B_{\epsilon/2}(f^{-k}(x))\cap M_{f^{-k}(x)})$$
for all $k\geq j$.

Let $\bar \tau=\max\{ \tau, \tau_s\}$.  Now arguing as in Proposition \ref{specforg} we can use $\bar \tau$ as the transition time and show that $\bar \GGG_\sigma$ has specification at scale $\epsilon$ for $f$.

We now have a decomposition with specification and the Bowen property, and the pressure of obstructions is less than the pressure of $f$ restricted to $\Lambda$.  From Theorem \ref{t.generalM} we see that there is a unique equilibrium state.
\end{proof}

For applications to the above theorem we note that the results in \cite{FO1} satisfy the hypotheses of Theorem \ref{t.generalAttractor}.

\bibliography{EqStatesPH}{}
\bibliographystyle{plain}
\end{document}